\documentclass[11pt,a4paper]{amsart}

\usepackage{amssymb,mathrsfs,enumitem,amsmath,amsthm,bbold}
\usepackage[mathscr]{euscript}
\usepackage{color}

\usepackage[all]{xy}
\definecolor{Chocolat}{rgb}{0.36, 0.2, 0.09}
\definecolor{BleuTresFonce}{rgb}{0.215, 0.215, 0.36}
\definecolor{EgyptianBlue}{rgb}{0.06, 0.2, 0.65}

\usepackage[backend=bibtex,
    sorting=anyt,
    isbn=false,
    url=false,
    doi=false,
    maxbibnames=100
    ]{biblatex}
\usepackage{fourier}

\usepackage[colorlinks,final,hyperindex]{hyperref}
\hypersetup{citecolor=BleuTresFonce, urlcolor=EgyptianBlue, linkcolor=Chocolat}
\theoremstyle{plain}
\newtheorem*{itheorem}{Theorem}
\newtheorem{theorem}{Theorem}[section]
\newtheorem{corollary}[theorem]{Corollary}

\newtheorem{lemma}[theorem]{Lemma}
\newtheorem{proposition}[theorem]{Proposition}

\newtheorem{conjecture}[theorem]{Conjecture}

\theoremstyle{definition}

\newtheorem{remark}[theorem]{Remark}

\DeclareMathAlphabet{\pazocal}{OMS}{zplm}{m}{n}

\def\calD{\pazocal{D}}
\def\calE{\pazocal{E}}
\def\calF{\pazocal{F}}

\def\calI{\pazocal{I}}
\def\calJ{\pazocal{J}}

\def\calS{\pazocal{S}}

\def\sfA{\mathsf{A}}

\def\sfV{\mathsf{V}}
\def\sfW{\mathsf{W}}
\def\LieH{\mathfrak{h}}
\def\sfx{\mathsf{x}}
\def\sfy{\mathsf{y}}
\def\sfz{\mathsf{z}}

\def\L{\mathrm{L}}
\def\T{\mathrm{T}}

\DeclareMathOperator{\pf}{Pf}

\DeclareMathOperator{\st}{st}

\DeclareMathOperator{\Der}{Der}

\DeclareMathAlphabet{\mathbbold}{U}{bbold}{m}{n}

\def\k{\mathbbold{k}}

\title[Young's lattice, invariant ideals, and weak identities]{Young's lattice, invariant ideals, and weak identities}

\author{Vladimir Dotsenko}
\address{Institute of Mathematics, Academy of Mathematics and Systems Science, Chinese Academy of Sciences, Beijing, 100190, China}
\address{Institut de Recherche Math\'ematique Avanc\'ee, UMR 7501, Universit\'e de Strasbourg et CNRS, 7 rue Ren\'e-Descartes, 67000 Strasbourg, France}
\email{vdotsenko@unistra.fr}

\author{Artem Lopatin}
\address{Department of Mathematics, Instituto de Matem\'atica, Estat\'istica e
Computa\c{c}\~ao Cient\'ifica (IMECC), Universidade Estadual de
Campinas (UNICAMP), Campinas, SP, Brazil}
\email{artem\_lopatin@yahoo.com}

\date{}

\begin{document}

\begin{abstract}
For a $d$-dimensional vector space $V$ over a field of characteristic zero, we show that $GL(V)$-stable ideals of the universal enveloping algebra of the free two-step nilpotent Lie algebra $\mathfrak n_2(V)$ are in natural correspondence with filters in Young's lattice of integer partitions with at most $d$ rows. The same classification holds for $GL(V)$-stable Poisson ideals of $S(V\oplus\Lambda^2(V))$ with its
natural Poisson bracket. As applications, we determine all weak polynomial identities of the $n$-th Weyl algebra relative
to its standard generating subspace and of the universal enveloping algebra of the $(2n+1)$-dimensional Heisenberg Lie algebra relative to
the Lie algebra itself. 
\end{abstract}

\maketitle
\section{Introduction}

\subsection{Littlewood identity and filters in Young's lattice}
In the theory of symmetric functions \cite{MR1354144}, there is a remarkable Schur function identity
 \[
\prod_{i}\frac{1}{1-x_i}\prod_{i<j}\frac{1}{1-x_ix_j}=\sum_{\lambda}s_\lambda(x),
 \]
where the sum on the right is over all integer partitions; in this form it seems to go back to Littlewood \cite[Sec.~11.9]{MR2127}, though an equivalent form of it was known already to Schur \cite{Schur1918}. Klyachko \cite[Th.~2]{Klyachko1978} found a beautiful categorification of this identity in the form of functorial isomorphisms, for finite-dimensional vector spaces $V$ over a field $\k$ of zero characteristic, of $GL(V)$-modules
 \[
U(\mathfrak{n}_2(V))\cong\bigoplus_{\lambda} S^{\lambda}(V), 
 \]
where the sum is over all integer partitions $\lambda$, $\mathfrak{n}_2(-)$ is the free two-step nilpotent Lie algebra, $S(-)$ is the symmetric algebra, and $S^\lambda(-)$ stands for the irreducible $GL(V)$-module corresponding to $\lambda$ (in fact, a thorough reader could perhaps infer this result already from older physics papers on parafermions such as~\cite{bracken1972}). This was later rediscovered by various authors, including Volichenko \cite[Prop.~2]{Volichenko_1984}, Drensky and Vladimirova \cite[Prop.~4]{Drensky_Vladimirova_1986}, and J\'{o}zefiak and Weyman~\cite{MR923671}; Gelfand and Zelevinsky \cite[p.~22]{MR757246} made an equivalent  observation, but replaced $U(\mathfrak{n}_2(V))$ by the isomorphic $GL(V)$-module $S(V\oplus\Lambda^2(V))$. 

Our first main result classifies $GL(V)$-stable ideals in $U(\mathfrak{n}_2(V))$ in terms of Young's lattice of integer partitions \cite{Kreweras1965}; in a sense, as $\dim(V)\to\infty$, it constructs a categorification of Young's lattice.
\begin{itheorem}[{Th.~\ref{th:A}}]
There is a natural bijection between $GL(V)$-stable ideals in  $U(\mathfrak{n}_2(V))$ and filters of Young's lattice restricted to partitions of length at most $\dim(V)$ (including the empty filter). 
\end{itheorem}

The same is true if one views $S(V\oplus\Lambda^2(V))$ as a Poisson algebra with the only nonzero Poisson brackets of generators given by $\{v_1,v_2\}:=v_1\wedge v_2$, and studies $GL(V)$-stable Poisson ideals of that algebra. 

We note that $GL(V)$-stable ideals of $S(V\oplus\Lambda^2(V))$ as a commutative associative algebra were classified by Daszkiewicz \cite{MR1018957}, following the same classification for $S(S^2(V))$ by Abeasis~\cite{MR602662} and for $S(\Lambda^2(V))$ by Abeasis and Del Fra~\cite{MR560133}. The $GL(V)$-decompositions of $S(S^2(V))$ and $S(\Lambda^2(V))$ correspond to two other identities of Littlewood:
\begin{gather*}
\prod_{i}\frac{1}{1-x_i^2}\prod_{i<j}\frac{1}{1-x_ix_j}=\sum_{\lambda=2\mu}s_\lambda(x),\\
\prod_{i<j}\frac{1}{1-x_ix_j}=\sum_{\lambda'=2\mu'}s_\lambda(x),
\end{gather*}
concerning the sum of all Schur functions corresponding to partitions into even parts and the sum of all Schur functions corresponding to partitions with even multiplicities of all parts. Interestingly enough, in both of these cases, \cite[Th.~3.1]{MR602662} and \cite[Th.~3.1]{MR560133} show that $GL(V)$-stable ideals of the symmetric algebras are in natural bijection with filters of the corresponding sublattices of Young's lattice, while in the case of $S(V\oplus\Lambda^2(V))$ the lattice of $GL(V)$-stable ideals is more complicated: in addition of comparing partitions in Young's lattice, one has to additionally compare the number of odd parts in the conjugate partition \cite[Th.~4.2]{MR1018957}. Our approach gives a meaningful way to remove this discrepancy. 

\subsection{Weak polynomial identities of Weyl and Heisenberg algebras}
As a consequence of our main result, we establish a result about weak polynomial identities of associative algebras. We recall that, given a pair  $(\sfA,\sfV)$, where $\sfA$ is an associative algebra and $\sfV$ is a vector subspace of $\sfA$, an associative polynomial $f(x_1,\ldots,x_n)$ is said to be a \emph{weak polynomial identity} of that  pair if for all $v_1,\ldots,v_n\in\sfV$, the evaluation $f(v_1,\ldots,v_n)$ vanishes in $\sfA$. 

In recent work of the second-named author, Rodriguez Palma and Tang, for the first Weyl algebra $\sfA_1$ and the vector space $\sfV_1$  spanned by its standard generators, a conjectural description \cite[Conj.~1.1]{Lopatin_Rodriguez_Tang_2025} of all weak polynomial identities of the pair $(\sfA_1,\sfV_1)$ was proposed.  In this paper, we describe, for all $n\ge 1$, all weak polynomial identities of the pair $(\sfA_n,\sfV_n)$, where $\sfA_n$ is the $n$th Weyl algebra and $\sfV_n$ is the span of its standard generators, and the ideal of weak identities of the pair $(U(\LieH_{2n+1}),\LieH_{2n+1})$, where $\LieH_{2n+1}$ is the Heisenberg Lie algebra; both pairs are defined explicitly in Section~\ref{sec:pairs}. To state our second main result, we recall that the standard associative polynomial $\st_{k}(x_1,\ldots,x_{k})$ is defined by the formula 
 \[ 
\st_{k}(x_1,\ldots,x_{k})=\sum_{\sigma\in S_k}(-1)^\sigma x_{\sigma(1)}\cdots x_{\sigma(k)}. 
 \]
\begin{itheorem}[{Th.~\ref{th:B}}]
For each $n\ge 1$, 
\begin{enumerate}
\item all weak polynomial identities of the pair $(\sfA_n,\sfV_n)$ follow from the identities $[[x_1,x_2],x_3]=\st_{2n+1}(x_1,\ldots,x_{2n+1})=0$, 
\item all weak polynomial identities of the pair $(U(\LieH_{2n+1}),\LieH_{2n+1})$ follow from the identities $[[x_1,x_2],x_3]=\st_{2n+2}(x_1,\ldots,x_{2n+2})=0$.    
\end{enumerate}
\end{itheorem}

It is interesting to note that from the point of view of the theory of identities of Lie algebras \cite{MR4368864}, it is not possible to distinguish between $\LieH_{2n+1}$ and $\LieH_{2m+1}$ for $m\ne n$: in both cases, all Lie identities follow from the two-step nilpotence identity $[[x_1,x_2],x_3]=0$. However, weak identities allow one to distinguish between these Lie algebras. 

While Theorem~\ref{th:B} led us to the discovery of the main result of this paper, it ended up being an example of how that main result can be applied to studying weak identities, and there are other ways of proving it. In particular, soon after the first version of our paper appeared on the \texttt{arXiv}, Futorny, Koshlukov and Schwarz \cite{FKS2026} circulated a proof of Theorem~\ref{th:B}(1) over an arbitrary infinite field by different methods. Furthermore, Theorem~\ref{th:B}(2) is close to that of Drensky \cite[Th.~2.4]{MR1451682} who studied ideals of identities stable under Lie substitutions. 

\subsection{Methods used in this work}

Proofs of both theorems stated above are of representation-theoretic nature, but they use two complementary approaches to representations of general linear groups on $U(\mathfrak{n}_2(V))$. For the first theorem, we construct the highest weight vectors using Pfaffians of skew-symmetric matrices. This way, the adjoint action of $\mathfrak{n}_2(V)$ can be implemented using differential operators acting on polynomial rings, and that action can be studied via the useful technique of harmonic analysis of introducing a scalar product for which operators of multiplication by a coordinate and partial derivatives become adjoint to each other \cite[Ch.~3]{MR754767}. For the second theorem, we construct highest weight vectors using standard polynomials in free associative algebras, and carefully investigate their evaluations. 

\subsection{Broader context of our work}
Historically, the main motivation to study invariant ideals in the abovementioned symmetric algebras came from invariant theory: for an algebraic group $G$ acting on an affine variety $M$, information about $G$-stable ideals in the algebra of regular functions $\k[M]$ sheds light on questions of algebra and geometry of closures of $G$-orbits in $M$ (in our particular cases, the $G$-action on $\k[M]$ is multiplicity-free, and after extension to an algebraic closure of $\k$, the action on $M$ has finitely many orbits~\cite{MR320173}). More recently, it became apparent that an explicit description of the lattice of $GL(V)$-stable ideals of a twisted commutative algebra \cite{SamSnowden2012}, that is, a commutative algebra object in polynomial functors of $V$, is a first step towards proving the Noetherian property of the category of $A(V)$-modules $M(V)$ with a compatible $GL(V)$-action, see, e.g., \cite[Sec.~4.1]{MR3477338}. That Noetherian property was proved for all twisted commutative algebras finitely generated by polynomial functors of degree $1$ by Snowden \cite{MR3018955}, for $S(S^2(\mathbb C^\infty))$ and $S(\Lambda^2(\mathbb C^\infty))$ by Nagpal, Sam, and Snowden \cite[Th.~1.1]{MR3477338}, and for $S(\mathbb C^\infty\oplus S^2(\mathbb C^\infty))$ and $S(\mathbb C^\infty\oplus\Lambda^2(\mathbb C^\infty))$ by Sam and Snowden \cite[Th.~1.1]{Sam_Snowden_2022}. The latter result is particularly notable, since it can be used to show that the category of polynomial $GL(\infty)$-equivariant modules over $U(\mathfrak n_2(\mathbb C^\infty))$ is locally Noetherian \cite[Th.~1.2]{Sam_Snowden_2022}. Our Theorem~\ref{th:A} complements that result by explicitly classifying the
$GL(V)$-stable ideals of $U(\mathfrak n_2(V))$ and the $GL(V)$-stable Poisson ideals of $S(V\oplus\Lambda^2(V))$ through Young-diagram filters. As a byproduct, our work builds a new bridge between the theory of twisted algebras and the theory of weak identities in algebras, which we believe to be of independent interest.

\subsection*{Structure of the paper} The paper is organised as follows. In Section \ref{sec:twisted}, we recall basic definitions of the theory of weak identities in algebras and of the theory of twisted algebras, for any variety of algebras, and highlight the little known fact that those two theories essentially study the same objects. In Section~\ref{sec:stable}, we prove Theorem~\ref{th:A}. In Section~\ref{sec:pairs}, we deduce from it Theorem~\ref{th:B}, and establish its corollaries for cocharacters and codimensions of weak identities; interestingly enough, the codimensions of weak identities for the pair $(U(\LieH_{3}),\LieH_{3})$ are given by Motzkin numbers. Finally, in Section \ref{sec:open}, we discuss several open problems naturally arising from our work.

\subsection*{Acknowledgements}
We thank Plamen Koshlukov for useful discussions. 
This paper was written while Artem Lopatin was visiting the Sino-Russian Mathematics Center at Peking University,
Beijing, China; he thanks this institution for hospitality and excellent working conditions. 

\subsection*{Statement on the AI use} 
\texttt{ChatGPT} was used to write programs for \texttt{Magma} which performed computational experiments at the discovery stage, as well as to check validity of written proofs and to catch typos in the final draft.

\subsection*{Conventions}
All vector spaces in this work are defined over a field $\k$ of zero characteristic. For vectors $v_1,\ldots,v_r$, their linear span is denoted by $\k\{v_1,\ldots,v_r\}$. We denote by $\Pi$ the set of non-negative integer \emph{partitions}, that is, non-increasing sequences of integers
$\lambda=(\lambda_1\ge\ldots\ge\lambda_k>0)$; for $d\ge0$, we denote $\Pi_d=\{\lambda\in\Pi:\ell(\lambda)\le d\}$. A filter in $\Pi_d$ is a subset closed upward under containment of Young diagrams. For a partition $\lambda$, we put $|\lambda|=\lambda_1+\cdots+\lambda_k$ and $\ell(\lambda)=k$; in particular, we allow $\ell(\lambda)=0$, in which case $\lambda=\varnothing$ is the empty partition and $|\varnothing|=0$. If $|\lambda|=n$, we write $\lambda\vdash n$.  For a partition $\lambda$, we may define the \emph{conjugate} partition $\lambda'$ with $\lambda'_i:=|\{p\colon\lambda_p\ge i\}|$. For two partitions $\lambda=(\lambda_1\ge\lambda_2\ge\ldots\ge\lambda_k)$ and $\mu=(\mu_1\ge\ldots\ge\mu_l)$, we say $\lambda\subseteq\mu$ if we have $k\le l$ and $\lambda_i\le\mu_i$ for all $i=1,\ldots,k$. We write $[a,b]$ for the commutator $ab-ba$ in an associative algebra, and $\{a,b\}$ for the bracket in a Poisson algebra.

\section{Twisted algebras and weak polynomial identities}\label{sec:twisted}

In this section, we discuss a relationship between the theory of twisted algebras and the theory of weak identities in algebras; this relationship has remained mostly unnoticed for decades, and we believe that recording it precisely is of independent interest. For simplicity, we concern ourselves with the case of algebras belonging to some variety $\mathfrak{M}$ of nonassociative algebras, that is, the category of $\k$-linear algebras with one binary operation stable under taking subalgebras, products, and homomorphic images; by the so-called HSP theorem of Birkhoff \cite[Th.~11.9]{MR648287}, this means that we consider the category of algebras satisfying a system of identities on the binary operation, and our assumption on the characteristic of the ground field implies that all these identities can be chosen multilinear \cite{MR668355}. Results of this section apply without any change to algebras with several not necessarily binary operations satisfying multilinear identities, that is, algebras over an operad \cite{MR2954392}.  

\subsection{Twisted algebras}

In recent years, the notion of a twisted commutative [associative] algebra has attracted quite a bit of attention due to its applications in questions of commutative algebra and algebraic geometry; see the paper of Erman, Sam and Snowden \cite{MR4300226} for one beautiful example. For a good introduction to this notion, we suggest the comprehensive survey of Sam and Snowden~\cite{SamSnowden2012}. Twisted algebras of other types have been studied before: the earliest reference known to us is the work of Barratt \cite{MR513566} on twisted Lie algebras (and twisted associative algebras as their universal envelopes); later Joyal \cite{MR927763} indicated that the key definitions are available for an arbitrary variety of algebras. In a sense, the notion is almost present in the well-known paper by Kelly \cite{MR2177746} written in 1971 but properly published only in 2005: it discusses the Day convolution of symmetric sequences, which is precisely the symmetric monoidal structure needed to define twisted algebras. While twisted algebras and coalgebras have been abundant in research in algebraic topology (e.g. configuration spaces), beyond the commutative case, they only appear in a handful of places in algebra, and even when they do appear, they  usually appear implicitly, see, e.g. the work of Etingof and Pak \cite{MR2390493} on the MacMahon Master Theorem.

Let us recall three equally useful viewpoints of twisted algebras for any variety $\mathfrak{M}$ of $\k$-linear algebras; see \cite[Sec.~5,6, 8.1]{SamSnowden2012} for a discussion of these in the commutative associative case. 

The first viewpoint regards a twisted algebra as a family of $\mathfrak{M}$-algebras $A(V)$ that depend on a vector space $V$ and are constructed out of  $V$ in a particular canonical way. Concretely, we require that the underlying vector space of $A(V)$ is a direct sum of homogeneous polynomial functors \cite[Sec.~I.A]{MR1354144}, and the algebra operations are natural transformations. The examples $\mathfrak n_2(V)$, $U(\mathfrak n_2(V))$, and $S(V\oplus\Lambda^2(V))$ belong to this category. They give twisted Lie, associative, and Poisson algebras, respectively; the last is also a twisted commutative associative algebra if one ignores its Poisson bracket, and was considered as such in \cite{Sam_Snowden_2022}.

For the second viewpoint, a twisted algebra is an $\mathfrak{M}$-algebra $A$ equipped with an action of the infinite general linear group $GL(\infty,\k)=\bigcup_{n\ge 1}GL(n,\k)$ by algebra automorphisms, so that as a $GL(\infty,\k)$-module, $A$ is a direct sum of polynomial representations, subquotients of direct sums of tensor powers of $V_\infty=\bigcup_{d\ge1}\k^d$. A twisted algebra in the sense of the previous paragraph gives us a twisted algebra in this sense, if we take $V=V_\infty$.

Finally, the third viewpoint suggests that a twisted algebra is an $\mathfrak{M}$-algebra $A$ which has a grading compatible with the product, so that the algebra structure is given by maps $A_n\otimes A_m\to A_{n+m}$, and each $A_n$ is given a right action of the symmetric group $S_n$ for which the product is equivariant with respect to the subgroups $S_n\times S_m\subset S_{n+m}$. Relating this approach to the first one goes via the Schur--Weyl duality between representations of symmetric groups and general linear groups~\cite[Sec.~5.18]{MR2808160}. Concretely, for a twisted associative algebra $A(V)$, one can consider its ``multilinear parts'' $A_m$ defined by
 \[
A_m:=A(V_\infty)_{1^m},   
 \]
where the right-hand side denotes the subspace of $A(\k^m)\subset A(V_\infty)$ consisting of eigenvectors of eigenvalue $t_1\cdots t_m$ of the diagonal matrix in $GL(m,\k)\subset GL(\infty,\k)$ with entries $t_1,\ldots,t_m$. These multilinear parts have the necessary $S_n\times S_m$-equivariant products $A_n\otimes A_m\to A_{n+m}$, and one can reconstruct the original twisted algebra by the formula 
 \[
A(V)=\bigoplus_{m\ge 0}A_m\otimes_{\k S_m}V^{\otimes m},  
 \]
see \cite[Sec.~I.A]{MR1354144}.

\subsection{Weak polynomial identities}

Let us recall basic notions of the theory of weak polynomial identities for subspaces of algebras. This goes back to the work of
Razmyslov~\cite{Razmyslov_1973_AL,Razmyslov_1973_USSR}, who used weak polynomial identities to study identities and central polynomials of matrix algebras. Strictly speaking, he was only considering the case where the algebra $\sfA$ is associative and the subspace $\sfV$ is a Lie subalgebra of $\sfA$ with respect to the operation $[a,b]=ab-ba$, but these constraints are not necessary to impose. In the fifty years since weak identities were introduced, many authors studied this notion in various contexts (nonassociative algebras, superalgebras, ground field of positive characteristic, finite ground field etc.). We do not attempt to give a detailed account of that work, and invite the reader to consult the recent survey of Drensky \cite{Drensky_2021} for an excellent summary of numerous results in this research area. 

Let $(\sfA,\sfV)$ be a pair consisting of an algebra $\sfA$ of some variety $\mathfrak{M}$ (Lie, associative, commutative,\ldots) and a subspace $\sfV\subseteq\sfA$. Let us denote by $F_{\mathfrak{M}}\langle X\rangle$ the free algebra of the same variety with the generating set~$X$. An element $f\in F_{\mathfrak{M}}\langle X\rangle$ is said to be a \emph{weak polynomial identity} (or simply a \emph{weak identity}) of the pair $(\sfA,\sfV)$ if for every function $\alpha\colon X\to\sfV$ the induced evaluation map $\alpha_*\colon F_{\mathfrak{M}}\langle X\rangle\to \sfA$ satisfies $\alpha_*(f)=0$. 

For the reader who has never considered weak identities, an instructional example is given by the polynomial $[x_1^2,x_2]=x_1^2x_2-x_2x_1^2$ in the free associative algebra generated by $x_1$ and $x_2$. This polynomial is well known to be a weak identity for the pair $(M_2(\k),\mathfrak{sl}_2(\k))$ consisting of all $2\times2$-matrices and its subspace of trace-zero matrices. Indeed, the square of a $2\times 2$-matrix with zero trace is known to be a scalar matrix, and hence $[x_1^2,x_2]=0$ whenever $x_1,x_2\in M_2(\k)$ satisfy $\mathrm{tr}(x_1)=0$. However, for some choices of $x_1,x_2\in M_2(\k)$, we have $[x_1^2,x_2]\ne 0$: for instance, one can take the standard matrix units $x_1=E_{11}$, $x_2=E_{12}$.

A central definition in the theory of weak identities is that of an \emph{$\mathsf{L}$-ideal}. By definition, an $\mathsf{L}$-ideal of the free algebra $F_{\mathfrak{M}}\langle X\rangle$ is an ideal stable under all algebra endomorphisms induced by linear endomorphisms of the vector space spanned by~$X$. 

\begin{remark}
We note that the work of Volichenko \cite{Volichenko_1984} mentioned in the introduction uses the term ``$\L$-ideal'' for a different notion: he means ideals that are stable under substitutions of arbitrary Lie elements. In particular, his classification result \cite[Th.~1]{Volichenko_1984} concerns ideals that are stable under substitutions of arbitrary Lie elements and additionally are stable under all derivations $\partial/\partial\xi$, $\xi\in V^*$, that send $v\in V$ to $\xi(v)1$.
\end{remark}

Let us record a simple result that asserts that different notions discussed in this section are in fact the same. 

\begin{proposition}\label{prop:equiv-weak-twist}
For $X=\{x_i\}_{i\ge 1}$, the following types of ideals in the free algebra $F_{\mathfrak{M}}\langle X\rangle$ coincide:
\begin{enumerate}
\item $GL(\infty,\k)$-stable ideals of $F_{\mathfrak{M}}\langle X\rangle$.
\item $\L$-ideals of $F_{\mathfrak{M}}\langle X\rangle$,
\item ideals of $F_{\mathfrak{M}}\langle X\rangle$ consisting of all weak identities of some pair $(\sfA,\sfV)$.
\end{enumerate}
\end{proposition}

\begin{proof}
Suppose that $J$ is a $GL(\infty,\k)$-stable ideal of $F_{\mathfrak{M}}\langle X\rangle$, that $f\in J$, and that $\phi\colon V_\infty\to V_\infty$ is a linear map. We note that $f$ depends on finitely many generators, and the images of those generators under $\phi$ involve finitely many generators, so to prove that the image of $f$ under $\phi$ belongs to $J$, we may restrict ourselves to the subalgebra $F_{\mathfrak{M}}\langle x_1,\ldots,x_n\rangle$ for some $n$. Considering the action of scalar matrices in $GL_n$, we see that a stable ideal must be homogeneous, that is, together with each element $f$, it contains all its homogeneous components individually. Our assumption means that  for each homogeneous component of degree $d$, $J_d$ is a $GL_n(\k)$-submodule of $F_{\mathfrak{M}}\langle x_1,\ldots,x_n\rangle_d$. Note that for each matrix $M\in M_n(\k)$, the action of $M+tI$ on the homogeneous component of the free algebra depends polynomially on $t$. For all but finitely many $t\in\k$, the matrix $M+tI$ is invertible,  and for such values the action of this matrix preserves $J_d$; interpolation recovers the value at $t=0$, which therefore lies in $J_d$ as well. 

Next, suppose that $J$ is an $\L$-ideal of $F_{\mathfrak{M}}\langle X\rangle$. Taking $\sfA=F_{\mathfrak{M}}\langle X\rangle/J$ and defining $\sfV$ as the subspace spanned by the image of $X$ in $\sfA$, we see that $J$ is the ideal of weak identities of $(\sfA,\sfV)$.

Finally, suppose that $J$ is the ideal of weak identities of some pair $(\sfA,\sfV)$. Then the action of $GL(\infty,\k)$ on an element of $J$ amounts, after evaluating on elements of $\sfV$, to replacing arguments by their linear combinations; those still belong to the subspace $\sfV$, and therefore the evaluation vanishes. 
\end{proof}

Another important notion in the theory of weak identities is that of \emph{cocharacters} and \emph{codimensions} of ideals of weak identities. These notions arise naturally from the third viewpoint of twisted algebras. Specifically, if we denote by $\operatorname{Id}(\sfA,\sfV)$ the ideal of $F_{\mathfrak{M}}\langle X\rangle$ consisting of weak identities of some pair $(\sfA,\sfV)$, the $m$-th weak codimension and the $m$-th weak cocharacter of the pair $(\sfA,\sfV)$ are defined by
 \[
c_m(\sfA,\sfV):=\dim\bigl(\left[F_{\mathfrak{M}}\langle X\rangle/\operatorname{Id}(\sfA,\sfV)\right]_{1^m}\bigr) , \qquad \chi_m(\sfA,\sfV):=\chi\bigl(\left[F_{\mathfrak{M}}\langle X\rangle/\operatorname{Id}(\sfA,\sfV)\right]_{1^m}\bigr), 
 \]
where $\chi(-)$ denotes the character of an $S_m$-module.

\section{Stable Poisson and associative ideals}\label{sec:stable}

\subsection{Recollections}\label{subsec:recollections}
Let $V$ be a vector space with $\dim(V)=d<\infty$. We assume that the reader is fluent in the theory of finite-dimensional polynomial representations of the group $GL(V)$ \cite[Chapter 5]{MR2808160}. In particular, finite-dimensional polynomial representations are semisimple, and simple polynomial representations are indexed by $\lambda\in\Pi$ with $\ell(\lambda)\le d$. If we choose an ordered basis of $V$, a highest weight vector is a nonzero common eigenvector for the upper triangular subgroup of $GL_d$, and this vector is said to be of weight $\lambda$ if the diagonal matrix with entries $t_1,\ldots,t_d$ multiplies it by $t_1^{\lambda_1}\cdots t_d^{\lambda_d}$. Each finite-dimensional irreducible polynomial representation has a one-dimensional highest weight space and is generated by any nonzero vector in it. 
For $\lambda\in\Pi$, we shall denote by $\calE_\lambda$ the isotypic component of $S^{\lambda}(V)$ in $S(V\oplus\Lambda^2(V))$ and by $\widehat{\calE}_\lambda$ the isotypic component of $S^{\lambda}(V)$ in $U(\mathfrak{n}_2(V))$; the Schur function identity from the introduction implies that for $\ell(\lambda)\le d$ we have $\calE_\lambda\cong \widehat{\calE}_\lambda\cong S^\lambda(V)$ and $\calE_\lambda=\widehat{\calE}_\lambda=0$ otherwise.

\subsection{Pfaffians and highest weight vectors} 

Let us choose a basis $e_1,\ldots, e_d$ of $V$, and identify the symmetric algebra $S(V\oplus\Lambda^2(V))$ with the polynomial algebra in variables $x_i$, $1\le i\le d$, and $y_{rs}$, $1\le r,s\le d$ satisfying $y_{sr}=-y_{rs}$ and $y_{rr}=0$; thus, the independent variables are $x_i$, $1\le i\le d$, and $y_{rs}$, $1\le r<s\le d$. In order to study $GL(V)$-stable ideals of this algebra, we shall need to access highest weight vectors of all its irreducible constituents; thus, we should exhibit one polynomial $w_\lambda$ for each partition $\lambda\in\Pi$, $\ell(\lambda)\le d$. These polynomials are defined as follows. For one-column partitions $\lambda=(1^l)$, $1\le l\le d$, we set
\begin{equation}\label{eq:Pfaff}
\begin{cases}
\qquad \qquad \left(\sum\limits_{1\le r<s\le l} y_{r,s}e_r\wedge e_s\right)^{\wedge m}
=:m!w_{1^l}e_1\wedge\cdots\wedge e_{l},\qquad  \text{ if } l=2m,\\
\left(\sum\limits_{i=1}^l x_ie_i\right)\wedge  
\left(\sum\limits_{1\le r<s\le l} y_{r,s}e_r\wedge e_s\right)^{\wedge m}
=:m!w_{1^l}e_1\wedge\cdots\wedge e_{l}, \quad \text{ if } l=2m+1.
\end{cases}     
\end{equation}
For general $\lambda$, we define 
 \[
w_\lambda=\prod_{j=1}^{\lambda_1}w_{1^{\lambda'_j}}, 
 \]
where for $\lambda=\varnothing$ we put $w_\lambda=1$.
 
The fact that these polynomials are highest weight vectors of  weight $\lambda$ is obvious: under our convention, the action of $GL(V)$ on $x_i$ is given by its action on $e_i$, and the action on $y_{ij}$ is given by its action on $e_i\wedge e_j$, so the action on $w_{1^l}$ comes from the action on $e_1\wedge\cdots\wedge e_{l}$, which is a highest weight vector of weight $(1^l,0^{d-l})$. Products of these give a highest weight vector of weight $\lambda$, since the polynomial algebra has no zero divisors.

These formulas are given by classical polynomials known as \emph{Pfaffians} of skew-symmetric matrices; our formulas can be interpreted as
 \[
w_{1^l}=
\begin{cases}
\qquad \qquad \pf(y_{rs})_{1\le r,s\le l},\qquad \qquad \quad l=2m,\\
 \sum\limits_{i=1}^l(-1)^{i-1}x_i\pf(y_{rs})_{1\le r,s\ne i\le l}, \qquad l=2m+1 .
\end{cases}
 \]
We refer the reader to \cite[App.~D]{MR1639468} for basic properties of Pfaffians.

\subsection{Main theorem}

In this section, we shall classify $GL(V)$-stable ideals in $U(\mathfrak{n_2}(V))$ and all $GL(V)$-stable Poisson ideals in the algebra $S(V\oplus\Lambda^2(V))$ equipped with the Poisson structure for which the only non-zero Poisson brackets of generators are given by $\{v_1,v_2\}:=v_1\wedge v_2$ for $v_1,v_2\in V$. The main nontrivial step of the proof is made by the following proposition. 

\begin{proposition}\label{prop:Poisson}
Let $\lambda\in\Pi$ with $\ell(\lambda)\le d$, and denote by $I_\lambda$ the $GL(V)$-stable Poisson ideal in $S(V\oplus\Lambda^2(V))$ generated by the highest weight vector $w_\lambda$. Then for every partition $\mu$ with $\ell(\mu)\le d$ and $\lambda\subseteq\mu$, we have $w_\mu\in I_\lambda$. 
\end{proposition}

\begin{proof}
The Schur modules $S^\lambda(V)$, the chosen highest weight vectors $w_\lambda$, and the Poisson algebra structure are defined over rationals, so it is enough to prove the assertion over $\mathbb Q$. Let us denote
 \[
\calD_\lambda:=V\calE_\lambda+\{V,\calE_\lambda\},   
 \]
where the first summand is the span of all products $vf$ with $v\in V$ and $f\in\calE_\lambda$, and the second summand is the span of all brackets $\{v,f\}$ with the same $v$ and $f$. It is enough to prove the assertion for $|\mu|-|\lambda|=1$, since the general case follows along a chain of one-box extensions.

The step $\varnothing\to(1)$ follows from $w_\varnothing=1$ and $w_{(1)}=x_1$. If a new last column is created in a nonempty diagram,
then $\mu=(\lambda_1+1,\lambda_2,\ldots)$ and 
 \[
w_\mu=x_1w_\lambda\in V\calE_\lambda\subset I_\lambda.   
 \]
We may therefore assume that an existing column is increased, and put $k=\lambda_1$.

Of course, the action of $V\oplus\Lambda^2(V)$ on $S(V\oplus\Lambda^2(V))$ by multiplications corresponds to the action of the generators $x_i$, $y_{rs}$ on the polynomial algebra by the same multiplications. Recall that for any Lie algebra $\mathfrak{g}$, and all $z,z_1,\ldots,z_q\in \mathfrak{g}$, the Poisson action of $z$ on $z_1\cdots z_q\in S^q(\mathfrak{g})$ is given by 
 \[
\{z,z_1\cdots z_q\}=\sum_{i=1}^qz_1\cdots z_{i-1}[z,z_i]z_{i+1}\cdots z_q,
 \]
so in our particular case, the Poisson action of $e_p$ corresponds, in the coordinates $x_i$, $y_{rs}$, to the action of the differential operator 
 \[
\sum_{i=1}^d y_{pi}\partial_{x_i}. 
 \]

The usual way of studying actions of differential operators on the polynomial ring \cite[Ch.~3]{MR754767} can be applied in our case. Let us define a bilinear form on $\mathbb{Q}[x_i,y_{jk}]$ by the formula
 \[
\langle P(x_i,y_{jk}),Q(x_i,y_{jk})\rangle =P(\partial_{x_i},\partial_{y_{jk}})[Q(x_i,y_{jk})] |_{x_i=y_{jk}=0}. 
 \]
By direct inspection, 
 \[
\left\langle \prod_{i}x_i^{p_i} \prod_{j<k}  y_{jk}^{p_{jk}},
\prod_{i}x_i^{q_i} \prod_{j<k} y_{jk}^{q_{jk}}\right\rangle =
\prod_{i}p_i!\delta_{p_i,q_i} \prod_{j<k}p_{jk}!\delta_{p_{jk},q_{jk}},
 \]
so this bilinear form is symmetric and positive definite, and the adjoints of multiplication operators are partial derivatives. 

The action of $GL(V)$ on $V\oplus\Lambda^2(V)$ induces an action on $S(V\oplus\Lambda^2(V))$. For our chosen coordinates, a matrix 
 $M\in GL_{d}$ acting on the variables $x_i$ gives rise to a matrix $\Lambda^2M$ acting on the variables $y_{jk}$, and more generally to an action on $\mathbb{Q}[x_i,y_{jk}]$ which we shall denote $\rho(M)$. Note that we have $\Lambda^2(M^t)=\Lambda^2(M)^t$, and more generally 
 \[
\langle \rho(M)(P),Q\rangle = \langle P, \rho(M^t)(Q)\rangle.  
 \]
In particular, if $U$ is a $GL(V)$-submodule, then its orthogonal complement with respect to our symmetric bilinear form is also a $GL(V)$-submodule. Since our symmetric bilinear form is positive definite over $\mathbb Q$, and the representation we consider is multiplicity-free, it follows that its distinct irreducible summands are mutually orthogonal.

The operator of multiplication by $e_p$ on $S(V\oplus\Lambda^2(V))$ corresponds to the multiplication of the polynomial ring by the variable $x_p$, and we noted that the Poisson action of $e_p$ corresponds to the action of the differential operator $\sum_{i=1}^d y_{pi}\partial_{x_i}$. Suppose that $w_\mu\notin\calD_\lambda$. Since $\calD_\lambda$ is a $GL(V)$-submodule of a multiplicity-free module, it then contains
no component of highest weight $\mu$. Orthogonality therefore gives
 \[
\langle e_p w_\lambda, w_\mu\rangle=\langle \{e_p,w_\lambda\},w_\mu\rangle=0  
 \]
for all $p=1,\ldots,d$. Since the corresponding adjoint operators are $\partial_{x_p}$ and 
 \[
\sum_{i\ne p}x_i\partial_{y_{pi}}=-\sum_{i\ne p}x_i\partial_{y_{ip}}, 
 \]
we conclude that 
 \[
\langle w_\lambda, \partial_{x_p}(w_\mu)\rangle=\langle w_\lambda,\sum_{i\ne p}x_i\partial_{y_{ip}}(w_\mu)\rangle=0  
 \]
for all $p=1,\ldots,d$.

The argument now consists of two steps. Suppose first that $\mu$ is obtained from $\lambda$ by increasing the part $\lambda_1'$:
 \[
\mu'=(\lambda'_1+1,\ldots,\lambda'_k). 
 \]
Let us take $p=\lambda'_1+1$. The element $w_\mu$ is equal to the product of $w_{1^p}$ and elements $w_{1^{\lambda_2'}}$, \ldots, $w_{1^{\lambda_k'}}$ that only depend on variables $x_i$ and $y_{rs}$ with $1\le i,r,s<p$. Therefore, we have
\begin{gather*}
\partial_{x_p}(w_\mu)=\partial_{x_p}(w_{1^p}) w_{1^{\lambda_2'}}\cdots w_{1^{\lambda_k'}},\\
\sum_{i\ne p}x_i\partial_{y_{ip}}(w_\mu)=\sum_{i\ne p}x_i\partial_{y_{ip}}(w_{1^p})w_{1^{\lambda_2'}}\cdots w_{1^{\lambda_k'}}.
\end{gather*}
If $p$ is odd, then our definition of $w_{1^p}$ immediately implies that $\partial_{x_p}(w_{1^p})=w_{1^{p-1}}=w_{1^{\lambda_1'}}$, and therefore $\partial_{x_p}(w_\mu)=w_\lambda$, implying that
 \[
\langle w_\lambda, \partial_{x_p}(w_\mu)\rangle=\langle w_\lambda, w_\lambda \rangle \ne 0,
 \]
which is a contradiction. Suppose that $p$ is even. Then either directly from Formula \eqref{eq:Pfaff} or using the Pfaffian expansion \cite[Eq.~(D.1)]{MR1639468}, we can write
\begin{multline*}
\sum_{i\ne p}x_i\partial_{y_{ip}}(w_{1^p})=
\sum_{i\ne p}x_i\partial_{y_{ip}}\sum_{j\le p-1}y_{jp}(-1)^{j-1}
\pf(y_{rs})_{1\le r,s\ne j\le p-1}
\\=\sum_{j=1}^{p-1}(-1)^{j-1}x_j
\pf(y_{rs})_{1\le r,s\le p-1,\ r,s\ne j}= w_{1^{p-1}}=w_{1^{\lambda_1'}},
\end{multline*}
so we have
 \[
\langle w_\lambda, \sum_{i\ne p}x_i\partial_{y_{ip}}(w_\mu)\rangle=\langle w_\lambda, w_\lambda \rangle \ne 0,
 \]
which is a contradiction, hence $w_\mu\in\calD_\lambda$.

Suppose now that an existing column $s>1$ is increased, so that
\[
 \mu'=(\lambda'_1,\ldots,\lambda'_s+1,\ldots,\lambda'_k),
 \qquad \lambda'_s<\lambda'_{s-1}.
\]
Write $\lambda=\bar\lambda+\rho$ and $\mu=\bar\mu+\rho$, where
addition is componentwise and
$\rho'=(\lambda'_1,\ldots,\lambda'_{s-1})$. Thus
$\bar\lambda'=(\lambda'_s,\ldots,\lambda'_k)$, and $\bar\mu$ is
obtained by increasing its first column. Put $d'=\lambda'_s+1$ and
$V'=\k\{e_1,\ldots,e_{d'}\}$. We embed $GL(V')$ into
$GL(V)$ by letting it act identically on the remaining basis vectors.
The first step, applied to $V'$, gives
\[
 w_{\bar\mu}\in V'\calE'_{\bar\lambda}
                  +\{V',\calE'_{\bar\lambda}\},
\]
where $\calE'_{\bar\lambda}$ is the isotypic component of highest weight $\bar\lambda$ in $S(V'\oplus\Lambda^2(V'))$. Thus, 
\[
 w_\mu=w_{\bar\mu}w_\rho\in V'\calE'_{\bar\lambda}w_\rho
                  +\{V',\calE'_{\bar\lambda}\}w_\rho,
\]
and it is enough to show that 
 \[
V'\calE'_{\bar\lambda}w_\rho\subset V\calE_\lambda \quad \text{ and }\quad  \{V',\calE'_{\bar\lambda}\}w_\rho\subset \{V,\calE_\lambda\}.  
 \]
To establish the first of those assertions, we note that every column of $\rho$ has height at least $d'$, so each of its columns transforms according to the determinant representation under the action of $GL(V')$, and therefore 
\[
 M(w_\rho)=\det(M)^{s-1}w_\rho
\]
for $M\in GL(V')$, implying that for such $M$, we have
\[
 M(w_{\bar\lambda})w_\rho
   =\det(M)^{-(s-1)}M(w_{\bar\lambda}w_\rho)
   =\det(M)^{-(s-1)}M(w_\lambda)\in\calE_\lambda.
\]
Since $\calE'_{\bar\lambda}$ is spanned by $M(w_{\bar\lambda})$ for $M\in GL(V')$, this implies $V'\calE'_{\bar\lambda}w_\rho\subset V\calE_\lambda$, as required.

For the second assertion, we shall establish the formula
\begin{equation}\label{eq:column_poisson_central}
 \{x_i,w_{1^l}\}=0
\end{equation}
that is valid for $1\le i\le l\le d$. For even $l$, the polynomial $w_{1^l}$ involves only the Poisson-central
variables $y_{rs}$, so the result is trivial. For odd $l$, put
\[
 q_j=(-1)^{j-1}\pf(y_{rs})_{1\le r,s\le l,\ r,s\ne j}.
\]
The Pfaffian cofactor identity is $\sum_{j=1}^l y_{ij}q_j=0$ for $1\le i\le l$. Indeed, append to the skew-symmetric matrix $(y_{rs})_{1\le r,s\le l}$ a first row and column identical to its $i$th row and column. The resulting even-order Pfaffian is zero by alternation; expansion in its first row gives this identity. Since $w_{1^l}=\sum_jx_jq_j$, it proves \eqref{eq:column_poisson_central} in the odd case as well.

Now the claim $\{V',\calE'_{\bar\lambda}\}w_\rho\subset \{V,\calE_\lambda\}$ is obvious: Equation \eqref{eq:column_poisson_central} and the Leibniz rule give $\{V',w_\rho\}=0$, so for all $f\in \calE'_{\bar\lambda}$, we have $\{v,f\}w_\rho=\{v,fw_\rho\}\in\{V,\calE_\lambda\}$, as required.
\end{proof}

In order to convert this result into a result about ideals in $U(\mathfrak{n}_2(V))$, recall that for a Lie algebra $\mathfrak{g}$, we can equip $U(\mathfrak g)$ with its usual PBW filtration, in which every element of $\mathfrak g$ has filtration degree $1$, and we have an isomorphism 
 \[
\operatorname{gr}U(\mathfrak g)\cong S(\mathfrak g),
 \] 
which is compatible with the Poisson algebra structures. Moreover, there is a linear map
 $
\sigma\colon S(\mathfrak g)\to U(\mathfrak g),   
 $
called the symmetrization and given, for $g_1\cdots g_n\in S^n(\mathfrak g)$, by 
 \[
\sigma(g_1\cdots g_n)=\frac1{n!}\sum_{\sigma\in S_n}g_{\sigma(1)}\cdots g_{\sigma(n)}   
 \]
and satisfying $\operatorname{gr}\sigma(g_1\cdots g_n)=g_1\cdots g_n$.

To state the following corollary, let us denote by $\hat{w}_\lambda$ a nonzero highest weight vector of the only $GL(V)$-submodule in
$U(\mathfrak{n}_2(V))$ that is isomorphic to $S^\lambda(V)$. Note that for $\mathfrak g=\mathfrak n_2(V)$, the symmetrization map is $GL(V)$-equivariant, so we may choose $\hat w_\lambda=\sigma(w_\lambda)$.

\begin{corollary}\label{cor:Assoc}
Let $\lambda\in\Pi$ with $\ell(\lambda)\le d$, and denote by $J_\lambda$ the $GL(V)$-stable ideal in $U(\mathfrak{n}_2(V))$ generated by $\hat{w}_\lambda$. Then for every partition $\mu$ with $\ell(\mu)\le d$ and $\lambda\subseteq\mu$, we have $\hat{w}_\mu\in J_\lambda$. 
\end{corollary}

\begin{proof}
We use the isomorphism of the Poisson algebras 
 \[
\operatorname{gr}U(\mathfrak{n}_2(V))\cong S(\mathfrak{n}_2(V)).
 \]
Since $J_\lambda$ is an ideal containing $\hat{w}_\lambda$, $\operatorname{gr}J_\lambda$ is a $GL(V)$-stable Poisson ideal containing $w_\lambda$. Thus $\operatorname{gr}J_\lambda$ contains $I_\lambda$ and, by Proposition~\ref{prop:Poisson}, also contains $w_\mu$, and hence the isotypic component of $S^\mu(V)$. Since the associated graded construction is $GL(V)$-equivariant and polynomial representations of $GL(V)$ are semisimple, this implies that the ideal $J_\lambda$ also contains the isotypic component of $S^\mu(V)$, and, in particular, contains $\hat w_\mu$.
\end{proof}

\begin{remark}\label{rem:one-sided}
In fact, every $GL(V)$-stable left ideal of $U(\mathfrak n_2(V))$ is two-sided. To see this, let $\iota$ be the antiautomorphism of the tensor algebra $T(V)$ induced by reversal of words, so that $\iota(v)=v$ for $v\in V$. It is well defined on the quotient by the triple commutators and commutes with $GL(V)$. By multiplicity freeness, it preserves each irreducible component, and hence every $GL(V)$-stable subspace. If $J$ is a stable left ideal, $a\in J$, and $b\in U(\mathfrak n_2(V))$, then $\iota(b)\iota(a)\in J$. Applying $\iota$ gives $ab\in J$. Thus the same classification applies to stable left ideals. 
\end{remark}

We are now ready to prove the main result of this section. 

\begin{theorem}\label{th:A}
There is a natural bijection between $GL(V)$-stable ideals in the universal enveloping algebra $U(\mathfrak{n}_2(V))$ and filters of Young's lattice restricted to partitions of length at most $\dim(V)$ (including the empty filter). The same is true if one views $S(V\oplus\Lambda^2(V))$ as a Poisson algebra with the only nonzero Poisson brackets of generators given by $\{v_1,v_2\}:=v_1\wedge v_2$, and studies $GL(V)$-stable Poisson ideals of that algebra.
\end{theorem}

\begin{proof}
We claim that the two correspondences in the statement are
\begin{equation}\label{eq:filter_correspondence}
 \calF\longmapsto\bigoplus_{\lambda\in\calF}\calE_\lambda,
 \qquad
 \calF\longmapsto\bigoplus_{\lambda\in\calF}\widehat{\calE}_\lambda,
\end{equation}
where $\calF$ ranges over the filters in $\Pi_d$.

Every $GL(V)$-stable subspace is a sum of simple components. If it is a Poisson ideal in $S(V\oplus\Lambda^2(V))$, Proposition~\ref{prop:Poisson} makes its set of highest weights upward closed. If it is an ideal in $U(\mathfrak{n}_2(V))$,
Corollary~\ref{cor:Assoc} gives the same conclusion.

Conversely, let $\calF\subseteq\Pi_d$ be a filter. Pieri's formula \cite[Sec.~I.5]{MR1354144} gives
\[
 V\otimes S^\lambda(V)\cong
 \bigoplus_{\substack{\mu\in\Pi_d,\ \lambda\subseteq\mu\\
                      |\mu|=|\lambda|+1}}S^\mu(V).
\]
In the Poisson algebra, both multiplication by $V$ and the Poisson
bracket with $V$ are $GL(V)$-equivariant maps from
$V\otimes\calE_\lambda$. Hence their images contain only the
one-box extensions in this formula. The first sum
in~\eqref{eq:filter_correspondence} is therefore closed under both
operations. For any element $f$ of this sum, the identity
\[
 y_{ij}f=\{x_i,x_jf\}-x_j\{x_i,f\}
\]
shows that it is also closed under multiplication by every $y_{ij}$; since the variables $y_{ij}$ are Poisson-central, closure under Poisson brackets with $y_{ij}$ is an empty condition. Since the $x_i$ and $y_{ij}$ generate the polynomial algebra, the sum $\bigoplus_{\lambda\in\calF}\calE_\lambda$ is a Poisson ideal.

In the enveloping algebra, the left and the right multiplication give equivariant maps $V\otimes\widehat{\calE}_\lambda\to U(\mathfrak{n}_2(V))$ and
$\widehat{\calE}_\lambda\otimes V\to U(\mathfrak{n}_2(V))$. The same Pieri's formula shows
that the sum $\bigoplus_{\lambda\in\calF}\widehat{\calE}_\lambda$ is closed under multiplication by $V$ on either side. As $V$ generates
$U(\mathfrak n_2(V))$, that sum is a two-sided ideal.
\end{proof}

Let us record an immediate corollary of this result using a different viewpoint on twisted algebras.

\begin{corollary}\label{cor:twisted-invariant-id}
There is a natural bijection between $GL(\infty,\k)$-stable ideals in the universal enveloping algebra $U(\mathfrak{n}_2(V_\infty))$ and filters of Young's lattice (including the empty filter). The same is true if one views $S(V_\infty\oplus\Lambda^2(V_\infty))$ as a Poisson algebra with the only nonzero Poisson brackets of generators given by $\{v_1,v_2\}:=v_1\wedge v_2$, and studies $GL(\infty,\k)$-stable Poisson ideals of that algebra.
\end{corollary}

\begin{proof}
The argument is completely identical to that of Theorem \ref{th:A}: every highest weight vector is contained in a subalgebra generated by $x_1,\ldots,x_d$ for some~$d$, so the same proof applies.
\end{proof}

\section{Weak identities for Weyl and Heisenberg algebras}\label{sec:pairs}

All algebras in this section are unital associative algebras over~$\k$. As above, we denote the ideal of weak identities of a pair $(\sfA,\sfV)$ in the free algebra $\k\langle X\rangle$ with a countable set of generators by $\operatorname{Id}(\sfA,\sfV)$. For a set of elements $\calF$ of $\k\langle X\rangle$, we use the notation $\langle\calF\rangle_{\L}$ for the $\L$-ideal generated by $\calF$, that is, the ideal generated by all substitutions of linear combinations of generators in the elements of $\calF$. Throughout this section, sans-serif letters $\sfx_i,\sfy_i$ denote elements of concrete algebras, whereas italic letters $x_i$ are generators of the free associative algebra in which the weak polynomial identities are considered.

Let us recall the basics concerning the Weyl and the Heisenberg algebras, referring the reader to \cite[Sec.~4.6]{MR498740} for details. Let $n\ge1$. The $n$-th Weyl algebra $\sfA_n$ is generated by
$\sfx_1,\sfy_1,\ldots,\sfx_n,\sfy_n$ with relations
\[
[\sfy_i,\sfx_j]=\delta_{ij}1,\qquad
[\sfy_i,\sfy_j]=[\sfx_i,\sfx_j]=0
\quad(1\le i,j\le n),
\]
where $\delta_{ij}$ is the Kronecker delta. It is known that the elements $\sfx_1^{a_1}\cdots \sfx_n^{a_n}\sfy_1^{b_1}\cdots \sfy_n^{b_n}$ with $a_i,b_i\ge 0$ form a basis of $\sfA_n$. We shall denote 
\[
\sfV_n=\k\{\sfx_1,\sfy_1,\ldots,\sfx_n,\sfy_n\}\subset\sfA_n.
\]
The $(2n+1)$-dimensional Heisenberg Lie algebra $\LieH_{2n+1}$ has a basis
$\sfx_1,\sfy_1,\ldots,\sfx_n$, $\sfy_n$, $\sfz$ and the Lie brackets
\[
[\sfy_i,\sfx_j]=\delta_{ij}\sfz,\qquad
[\sfy_i,\sfy_j]=[\sfx_i,\sfx_j]=[\sfy_i,\sfz]=[\sfx_i,\sfz]=0.
\]
We shall identify $\LieH_{2n+1}$ with its image in $U(\LieH_{2n+1})$.\\

Our goal in this section is to prove the following theorem.

\begin{theorem}\label{th:B}
For each $n\ge 1$, 
\begin{enumerate}
\item the $\mathsf{L}$-ideal of weak polynomial identities of the pair $(\sfA_n,\sfV_n)$ is generated by the elements $[[x_1,x_2],x_3]$ and $\st_{2n+1}$, 
\item the $\mathsf{L}$-ideal of weak polynomial identities of the pair $(U(\LieH_{2n+1}),\LieH_{2n+1})$ is generated by the elements $[[x_1,x_2],x_3]$ and $\st_{2n+2}$.    
\end{enumerate}
\end{theorem}

We shall deduce it from the classification of $GL(\infty,\k)$-stable ideals in the universal enveloping algebra $U(\mathfrak{n}_2(V_\infty))$ given in Corollary \ref{cor:twisted-invariant-id}.   

\subsection{The highest weight vectors}

In this section, it will be convenient for us to choose different highest weight vectors that are better adapted for studying weak identities. The vector $\hat w_\lambda$ introduced before Corollary~\ref{cor:Assoc} was specified only up to a nonzero scalar.
For the rest of this section, we choose its normalization as follows.

\begin{lemma}\label{lem:pairs_highest}
For $\dim(V)=d$ and $\lambda\in\Pi_d$, the product
\begin{equation}\label{eq:pairs_highest}
\hat w_\lambda:=\prod_{j=1}^{\lambda_1}
 \st_{\lambda'_j}(x_1,\ldots,x_{\lambda'_j}),
\end{equation}
where for $\lambda=\varnothing$, the product is understood to be $1$, is a nonzero highest weight vector of the isotypic component~$\widehat{\calE}_\lambda$ in $U(\mathfrak{n}_2(V))$.
\end{lemma}

\begin{proof}
For $1\le k\le d$, alternation defines a $GL(V)$-equivariant map
\[
\Lambda^k(V)\longrightarrow U(\mathfrak{n}_2(V)),\qquad
v_1\wedge\cdots\wedge v_k\longmapsto\st_k(v_1,\ldots,v_k).
\]
Its value on $e_1\wedge\cdots\wedge e_k$ is nonzero. To see this, we map $U(\mathfrak{n}_2(V))$ to the exterior algebra $\Lambda(V)$ by sending every generator $v$ of $U(\mathfrak{n}_2(V))$ to the same element of the exterior algebra. This morphism of algebras is well defined because each element $[v,w]=2v\wedge w$ is central in $\Lambda(V)$. The image under this map of the standard
polynomial $\st_k$ is $k! e_1\wedge\cdots\wedge e_k\ne0$. This shows that each individual factor in Formula~\eqref{eq:pairs_highest} is a highest weight vector
of weight $(1^{\lambda'_j})$.

The algebra $U(\mathfrak{n}_2(V))$ has no zero divisors: its associated graded algebra for the PBW filtration is the polynomial algebra
$S(\mathfrak n_2(V))$. Hence the product in~\eqref{eq:pairs_highest} is nonzero. Since $GL(V)$ acts by algebra automorphisms, it is a
highest weight vector whose weight is the sum of the column weights, namely $\lambda$. Since the $GL(V)$-action on $U(\mathfrak{n}_2(V))$ is multiplicity-free, this highest weight vector is the highest weight vector of the only copy of $S^\lambda(V)$ in $U(\mathfrak{n}_2(V))$, that is, of $\widehat{\calE}_\lambda$.
\end{proof}

\subsection{The Weyl pair \texorpdfstring{$(\sfA_n,\sfV_n)$}{(An,Vn)}}

\begin{proof}[Proof of Theorem~\ref{th:B}(1)]
Write $\calI_n=\operatorname{Id}(\sfA_n,\sfV_n)$. It is clear that $[[x_1,x_2],x_3]$ and $\st_{2n+1}$ are weak identities of the pair $(\sfA_n,\sfV_n)$: for the first one, it follows from the fact that every commutator of linear combinations of generators is proportional to $1$, for the second one, one merely needs to note that we evaluate an alternating and multilinear operation $\st_{2n+1}$ on elements of a vector space of dimension~$2n$. This allows us to use Theorem \ref{th:A} and its Corollary \ref{cor:twisted-invariant-id}, converting information about $GL(\infty,\k)$-stable ideals of $U(\mathfrak{n}_2(V_\infty))$ containing $\st_{2n+1}$ into information about $\L$-ideals of free associative algebra containing $[[x_1,x_2],x_3]$ and $\st_{2n+1}$. Indeed, it follows from Corollary \ref{cor:twisted-invariant-id} that $\calI_n$ contains all isotypic components $\widehat{\calE}_\lambda$ where $\lambda'$ contains at least one part strictly greater than $2n$. Let us show that no other isotypic component belongs to $\calI_n$. For that, we shall exhibit a non-zero evaluation on $\sfV_n$ for every remaining highest weight vector \eqref{eq:pairs_highest} with $\lambda_1'\le 2n$.

Let us first note that for $1\le j\le n$,
\begin{equation}\label{eq:pairs_weyl_even}
\st_{2j}(\sfy_1,\sfx_1,\ldots,\sfy_j,\sfx_j)=j!\,1.
\end{equation}
Indeed, for arbitrary elements of an associative algebra,
\begin{equation}\label{eq:pairs_even_expansion}
\st_{2j}(a_1,\ldots,a_{2j})
 =\frac1{2^j}\sum_{\tau\in S_{2j}}(-1)^\tau
   [a_{\tau(1)},a_{\tau(2)}]\cdots
   [a_{\tau(2j-1)},a_{\tau(2j)}],
\end{equation}
since expanding the commutators produces $2^j$ copies of the standard polynomial. For the $2j$-tuple~\eqref{eq:pairs_weyl_even}, only the
$2^j j!$ terms pairing each $\sfy_i$ with $\sfx_i$ survive. Reversing a pair changes both its commutator and its permutation sign; permuting the pairs has positive sign. Every surviving signed product is $1$, which proves~\eqref{eq:pairs_weyl_even}.

To proceed, we remark that, if $b$ commutes with $a_1,\ldots,a_{2j}$, grouping the terms according to the position of $b$ yields
\begin{equation}\label{eq:pairs_odd_insertion}
\st_{2j+1}(a_1,\ldots,a_{2j},b)
 =\left(\sum_{i=0}^{2j}(-1)^i\right)
   \st_{2j}(a_1,\ldots,a_{2j})b
 =\st_{2j}(a_1,\ldots,a_{2j})b.
\end{equation}
Consequently, for $0\le j<n$,
\begin{equation}\label{eq:pairs_weyl_odd}
\st_{2j+1}(\sfy_1,\sfx_1,\ldots,\sfy_j,\sfx_j,\sfy_{j+1})=j!\sfy_{j+1}.
\end{equation}

Consider now a highest weight vector \eqref{eq:pairs_highest} with $\lambda_1'\le 2n$. We evaluate that noncommutative polynomial by 
assigning, for all $i\le\lambda_1'$,
\[
x_i\mapsto 
\begin{cases}
\sfy_{r}, \quad i=2r-1,\\
\sfx_{r}, \quad i=2r. 
\end{cases}
\]
Equations
\eqref{eq:pairs_weyl_even} and~\eqref{eq:pairs_weyl_odd} show that
\begin{equation}\label{eq:pairs_weyl_highest}
\hat w_\lambda\longmapsto
 \left(\prod_{s=1}^{\lambda_1}
       \left\lfloor\frac{\lambda'_s}{2}\right\rfloor!\right)
 \prod_{j=0}^{n-1}\sfy_{j+1}^{\,a_j},
\qquad
 a_j=\#\{s:\lambda'_s=2j+1\}.
\end{equation}
This is a non-zero element of $\sfA_n$, as required.
\end{proof}

\subsection{The Heisenberg pair
\texorpdfstring{$(U(\LieH_{2n+1}),\LieH_{2n+1})$}{(U(h2n+1),h2n+1)}}

\begin{proof}[Proof of Theorem~\ref{th:B}(2)]
Put $\calJ_n=\operatorname{Id}(U(\LieH_{2n+1}),\LieH_{2n+1})$. Since $[\LieH_{2n+1},\LieH_{2n+1}]=\k \sfz$ is central,
$[[x_1,x_2],x_3]$ is a weak identity; since we evaluate an alternating and multilinear operation $\st_{2n+2}$ on elements of a vector space $\LieH_{2n+1}$ of dimension~$2n+1$, $\st_{2n+2}$ is also a weak identity. It follows from Corollary \ref{cor:twisted-invariant-id} that $\calJ_n$ contains all isotypic components $\widehat{\calE}_\lambda$ where $\lambda'$ contains at least one part strictly greater than $2n+1$.

The same commutator expansion~\eqref{eq:pairs_even_expansion}, now
with $[\sfy_i,\sfx_i]=\sfz$, gives
\[
\st_{2j}(\sfy_1,\sfx_1,\ldots,\sfy_j,\sfx_j)=j!\sfz^j\ne0
\]
for $1\le j\le n$. Using~\eqref{eq:pairs_odd_insertion}, we also obtain
\[
\st_{2j+1}(\sfy_1,\sfx_1,\ldots,\sfy_j,\sfx_j,\sfy_{j+1})
 =j!\sfz^j\sfy_{j+1}\ne0
\]
for $0\le j<n$. Finally, since $\sfz$ is central, we have
\begin{equation}\label{eq:pairs_heisenberg_top}
\st_{2n+1}(\sfy_1,\sfx_1,\ldots,\sfy_n,\sfx_n,\sfz)=n!\sfz^{n+1}\ne0.
\end{equation}

Consider now a highest weight vector \eqref{eq:pairs_highest} with $\lambda_1'\le 2n+1$. We evaluate that noncommutative polynomial by 
assigning, for all $i\le\lambda_1'$,
\[
x_i\mapsto 
\begin{cases}
\sfy_{r}, \quad i=2r-1, \quad 1\le r\le n,\\
\sfx_{r}, \quad i=2r, \!\!\quad \qquad 1\le r\le n,\\
\sfz, \quad\,\,\,\,  i=2n+1.
\end{cases}
\]
The preceding evaluations show that
\begin{equation}\label{eq:pairs_heisenberg_highest}
\hat w_\lambda\longmapsto
 \left(\prod_{s=1}^{\lambda_1}
       \left\lfloor\frac{\lambda'_s}{2}\right\rfloor!\right)
 \sfz^c\prod_{j=0}^{n-1}\sfy_{j+1}^{\,a_j},
\end{equation}
where
\[
a_j=\#\{s:\lambda'_s=2j+1\}\quad(0\le j<n),\qquad
b=\#\{s:\lambda'_s=2n+1\},\qquad
c=\sum_{s=1}^{\lambda_1}
       \left\lfloor\frac{\lambda'_s}{2}\right\rfloor+b.
\]
These are non-zero elements of $U(\LieH_{2n+1})$, as required.
\end{proof}

\begin{remark}
As we mentioned in the introduction, the result of Theorem~\ref{th:B}(2) is closely related to the work of Drensky \cite[Th.~2.4]{MR1451682} who exhibited the same generators for the ideal of identities of the pair  $(U(\LieH_{2n+1}),\LieH_{2n+1})$ stable under all Lie substitutions. Indeed, Formula \eqref{eq:pairs_even_expansion} shows that the result of every nontrivial Lie substitution into $\st_{2n+2}$ belongs to the $\L$-ideal generated by $[[x_1,x_2],x_3]$ and hence can be ignored when we consider the $\L$-ideal generated by both $[[x_1,x_2],x_3]$ and $\st_{2n+2}$. 
\end{remark}

\subsection{Weak cocharacters and codimensions}
To describe the weak cocharacters and codimensions of the pairs $(\sfA_n,\sfV_n)$ and $(U(\LieH_{2n+1}),\LieH_{2n+1})$, we use the standard notation $\chi^\lambda$ and $d_\lambda=\chi^\lambda(1)$ for the respective character and dimension of the simple $S_m$-module $M_\lambda$ corresponding to the Young diagram $\lambda$ with $|\lambda|=m$, see \cite[Sec.~I.7]{MR1354144}.

\begin{corollary}\label{cor:pairs_codimensions}
Let $n,m\ge1$. We have
\begin{gather}
 \chi_m(\sfA_n,\sfV_n)
   =\sum_{\substack{\lambda\vdash m\\\ell(\lambda)\le2n}}
             \chi^\lambda,\qquad 
 c_m(\sfA_n,\sfV_n)
   =\sum_{\substack{\lambda\vdash m\\\ell(\lambda)\le2n}}d_\lambda,\\
 \chi_m(U(\LieH_{2n+1}),\LieH_{2n+1})
   =\sum_{\substack{\lambda\vdash m\\\ell(\lambda)\le2n+1}}
             \chi^\lambda,\qquad 
 c_m(U(\LieH_{2n+1}),\LieH_{2n+1})
   =\sum_{\substack{\lambda\vdash m\\\ell(\lambda)\le2n+1}}d_\lambda.
\end{gather}
\end{corollary}

\begin{proof}
This is an immediate consequence of Theorem~\ref{th:B} and Schur--Weyl duality \cite[Sec.~5.18]{MR2808160}: on the level of multilinear components, the simple $S_m$-modules that appear correspond precisely to the Young diagrams for which $S^\lambda(V)$ appears, with the same multiplicities.
\end{proof}

In particular, for $n=1$ the formulas specialize to
\begin{equation*}
 c_m(\sfA_1,\sfV_1)
   =\binom{m}{\lfloor m/2\rfloor},\qquad 
 c_m(U(\LieH_3),\LieH_3)
   =\sum_{j=0}^{\lfloor m/2\rfloor}
        \frac{1}{j+1}\binom{m}{2j}\binom{2j}{j}.
\end{equation*}
The first formula follows from $d_{(m-j,j)}=\binom{m}{j}-\binom{m}{j-1}$. Combining the identity of Regev \cite{Regev_1981} and the observation of Donaghey and Shapiro \cite{MR505544}, one recognizes the second number as the $m$th Motzkin number. 

\subsection{Another Weyl pair}

As another consequence of Theorem~\ref{th:B}(2), we describe weak
identities of the Weyl algebra relative to the generating subspace that includes the unit of the associative product. For every $n\ge1$, put
\[
\mathsf{H}_n=\k\{\sfx_1,\sfy_1,\ldots,\sfx_n,\sfy_n,1\}
\subseteq\sfA_n.
\]

\begin{corollary}\label{cor:weyl_heisenberg}
For every $n\ge1$,
\[
\operatorname{Id}(\sfA_n,\mathsf{H}_n)
 =\langle [[x_1,x_2],x_3],\st_{2n+2}\rangle_{\L}.
\]
\end{corollary}

\begin{proof}
For each $t\in\k^\times$,
the assignments
\[
\rho_t(\sfx_i)=\sfx_i,\qquad
\rho_t(\sfy_i)=t\sfy_i,\qquad \rho_t(\sfz)=t1
\]
define a surjective algebra homomorphism $\rho_t\colon U(\LieH_{2n+1})\to\sfA_n$.
Its restriction maps $\LieH_{2n+1}$ bijectively onto $\mathsf{H}_n$.
Consequently,
\[
\operatorname{Id}(U(\LieH_{2n+1}),\LieH_{2n+1})
\subseteq\operatorname{Id}(\sfA_n,\mathsf{H}_n).
\]
Conversely, the PBW basis of $U(\LieH_{2n+1})$ shows that evaluating a non-zero element of that algebra at $\sfz=t$ cannot be equal to zero for all $t\ne 0$, since $\k$ is infinite. Thus, every polynomial which is a weak identity after every specialization $\rho_t$, $t\ne 0$, is already a weak identity of $U(\LieH_{2n+1})$, proving the reverse inclusion, so Theorem~\ref{th:B}(2) applies.
\end{proof}

\section{Open problems arising from our results}\label{sec:open}

There are several promising directions related to our work. Let us outline three of them.

\subsection{Relating other stable ideals in \texorpdfstring{$U(\mathfrak{n}_2(V))$}{Un2V} to weak identities}

In Theorem~\ref{th:B}, we identified, for each $n\ge 3$, the $\L$-ideal generated by $[[x_1,x_2],x_3]$ and $\st_n$ as the weak ideal of a very concrete pair $(\sfA,\sfV)$. Let us give another result of this kind.

\begin{proposition}
Consider the Grassmann algebra $\Lambda(V)$ generated by a vector space $V$ of countably infinite dimension. The $\L$-ideal of weak polynomial identities of the pair $(\Lambda(V),\Lambda(V))$ is generated by the elements $[[x_1,x_2],x_3]$ and $[x_1,x_2]^2$. It corresponds to the $GL(V)$-stable ideal of $U(\mathfrak{n}_2(V))$ generated by $\widehat{\calE}_{\lambda}$ with $\lambda=(2^2)$.
\end{proposition}

\begin{proof}
For any algebra $A$, the ideal of weak identities of the pair $(A,A)$ is the $\T$-ideal of all identities of $A$.  
Krakowski and Regev in \cite{MR325658} show that the $\T$-ideal of identities of $\Lambda(V)$ is generated by the single element $[[x_1,x_2],x_3]$. On the other hand, this latter $\T$-ideal was previously considered by Latyshev who established in \cite{MR156874} that as a two-sided ideal, this $\T$-ideal is generated by all elements $[[x_i,x_j],x_k]$ and $[x_i,x_j][x_k,x_l]+[x_i,x_k][x_j,x_l]$, for $i,j,k,l\ge 1$. In terms of weak polynomial identities, this means that as an $\L$-ideal, this ideal is generated by $[[x_1,x_2],x_3]$ and $[x_1,x_2][x_3,x_4]+[x_1,x_3][x_2,x_4]$. It remains to note that substitution $x_3=x_1$, $x_4=x_2$ into $[x_1,x_2][x_3,x_4]+[x_1,x_3][x_2,x_4]$ gives $[x_1,x_2]^2$, and that the multilinearization of the polynomial $[x_1,x_2]^2$, obtained by replacing $x_1$ by $x_1+x_4$ and $x_2$ by $x_2+x_3$, and extracting multilinear terms, is
\begin{multline*}
[x_1,x_2][x_4,x_3]+[x_1,x_3][x_4,x_2]+[x_4,x_2][x_1,x_3]+[x_4,x_3][x_1,x_2]\\ =-2([x_1,x_2][x_3,x_4]+[x_1,x_3][x_2,x_4])+[[x_1,x_2],[x_3,x_4]]+[[x_1,x_3],[x_2,x_4]], 
\end{multline*}
which shows that the $\L$-ideal generated by $[[x_1,x_2],x_3]$ and $[x_1,x_2][x_3,x_4]+[x_1,x_3][x_2,x_4]$ coincides with the $\L$-ideal generated by $[[x_1,x_2],x_3]$ and $[x_1,x_2]^2$. It remains to note that $[x_1,x_2]^2$ is a highest weight vector of highest weight $(2,2)$, and therefore generates the $GL(V)$-submodule $\widehat{\calE}_{\lambda}$ with $\lambda=(2^2)$. 
\end{proof}

It would be interesting to find meaningful pairs $(\sfA,\sfV)$ corresponding to other $GL(V)$-stable ideals of $U(\mathfrak{n}_2(V))$, for instance, to the $GL(V)$-stable ideal generated by $\widehat{\calE}_\lambda$ for a given Young diagram $\lambda$. (Of course, we saw in the proof of Proposition \ref{prop:equiv-weak-twist} that exhibiting one such pair in a tautological way is always possible, hence the qualifier ``meaningful''.)

\subsection{Products of isotypic components in \texorpdfstring{$U(\mathfrak{n}_2(V))$}{Un2V}}

Another natural question arising from our work is parallel to the following question about multiplicity-free actions in harmonic analysis and representation theory. Suppose that a complex reductive algebraic group $G$ acts on an algebraic variety $M$ in a way that the $G$-action on $\mathbb{C}[M]$ is multiplicity-free. Suppose that $V_\lambda$, $V_\mu$, and $V_\nu$ are three irreducible representations of $G$ that occur as constituents in $\mathbb{C}[M]$ (necessarily with multiplicity one), and that $\calS_\lambda$, $\calS_\mu$, and $\calS_\nu$ are the corresponding isotypic components of $\mathbb{C}[M]$. Denote by $\calS_\lambda \calS_\mu$ the vector subspace of $\mathbb{C}[M]$ spanned by all products of elements of $\calS_\lambda$ and $\calS_\mu$. It is reasonable to ask whether $\calS_\nu\subset \calS_\lambda \calS_\mu$ if and only if $V_\nu$ embeds in the tensor product $V_\lambda\otimes V_\mu$. This question has been studied in \cite{MR4562114,MR2039124,MR2504020}, where some interesting results were obtained.

In the case $G=GL(V)$ and $M=V^*\oplus\Lambda^2(V^*)$, so that $\mathbb{C}[M]=S(V\oplus\Lambda^2(V))$, results of Daszkiewicz \cite{MR1018957} show that the corresponding statement is not true already for $|\lambda|=1$. By contrast, our Corollary \ref{cor:Assoc}, paired with Remark \ref{rem:one-sided}, implies that for $\min(|\lambda|,|\mu|)=1$ the expected result is true. However, it is not true in full generality, as the following example shows.

\begin{proposition}
For $\dim(V)\ge 3$, we have $\widehat{\calE}_{2,1^2}\not\subset \widehat{\calE}_{1^2}\widehat{\calE}_{1^2}$ even though $S^{2,1^2}(V)$ is isomorphic to a submodule of $S^{1^2}(V)\otimes S^{1^2}(V)$.
\end{proposition} 

\begin{proof}
The fact that $S^{2,1^2}(V)$ is isomorphic to a submodule of $S^{1^2}(V)\otimes S^{1^2}(V)$ follows from Pieri's formula \cite[Sec.~I.5]{MR1354144}. For the other claim, it is enough to note that $\widehat{\calE}_{1^2}=\Lambda^2(V)$ consists of central elements of $U(\mathfrak n_2(V))$, and therefore the multiplication map
 \[
\widehat{\calE}_{1^2}\otimes\widehat{\calE}_{1^2}\to U(\mathfrak n_2(V))   
 \]
factors through $S^2(\Lambda^2(V))\cong S^{2^2}(V)\oplus \Lambda^4(V)$, with no copies of $S^{2,1^2}(V)$; the latter module is non-zero precisely for $\dim(V)\ge 3$. 
\end{proof}

Recall that $S^\nu(V)$ is isomorphic to a submodule of $S^{\lambda}(V)\otimes S^\mu(V)$ if and only if $c_{\lambda,\mu}^\nu>0$, where $c_{\lambda,\mu}^\nu$ is the classical Littlewood--Richardson coefficient \cite[Sec.~I.9]{MR1354144}. It would be interesting to know which  conditions on $\lambda,\mu,\nu$ one needs to impose to ensure that $\widehat{\calE}_\nu\subset \widehat{\calE}_\lambda\widehat{\calE}_\mu$.

\subsection{Weak identities of the pair \texorpdfstring{$(\sfA_1,\sfW_1)$}{A1W1}}

The last conjecture we would like to advertise concerns a relationship of the standard identity $\st_3$ to a very particular subspace of the first Weyl algebra $\sfA_1$, the Witt Lie algebra $\sfW_1$. Classically, $\sfW_1$ is defined as the Lie algebra $\Der(\k[t])$ of polynomial vector fields on the affine line; it can be embedded into $\sfA_1$ as the linear span of all elements $\sfx_1^k\sfy_1$, $k\ge 0$. It was discovered by Sumenkov \cite{Sumenkov1976} that this Lie algebra satisfies the standard identity of degree five
 \[
\sum_{\sigma\in S_4}(-1)^\sigma
[[[[v_0,v_{\sigma(1)}],v_{\sigma(2)}],v_{\sigma(3)}],v_{\sigma(4)}]=0,   
 \]
and it is a famous open problem in the theory of polynomial identities in Lie algebras whether all identities of $\sfW_1$ follow from that identity, see \cite[Sec.~14.3]{MR4368864} for a discussion of this question.  We would like to propose a potentially more accessible conjecture in the context of weak identities, which we view as a promising first step in understanding Lie identities of $\sfW_1$. Namely, it is known that the associative standard polynomial $\st_3$ is an identity of the pair $(\sfA_1,\sfW_1)$, see, for instance, \cite[Cor.~5.5]{Dzhumadildaev_2004}. Computer experiments suggest the following conjecture.

\begin{conjecture}
The ideal of weak identities of the pair $(\sfA_1,\sfW_1)$ is generated by all elements $\st_3(f_1,f_2,f_3)$, where $f_1,f_2,f_3$ are Lie polynomials in the free generators.
\end{conjecture}

In the original terminology of Razmyslov \cite{Razmyslov_1973_AL}, this conjecture asserts that the standard identity $\st_3$ generates the ideal of weak identities of the pair $(\sfA_1,\sfW_1)$ as an ideal of associative--Lie pairs.

\printbibliography

\end{document}